\documentclass[12pt]{amsart}

\usepackage{amsmath}
\usepackage{amssymb}
\usepackage{amsthm}
\usepackage{braket}
\usepackage{enumerate} 
\usepackage{amsthm}
\usepackage{graphicx}
\usepackage{psfrag}
\usepackage{framed}
\usepackage[T1]{fontenc}
\input{xy}  
\xyoption{all} 
\CompileMatrices 

\theoremstyle{plain}
\newtheorem{theorem}{Theorem}[section]
\newtheorem*{theorem*}{Theorem}
\newtheorem{lemma}[theorem]{Lemma}
\newtheorem{corollary}[theorem]{Corollary}
\newtheorem{proposition}[theorem]{Proposition}
\newtheorem{conjecture}[theorem]{Conjecture}

\theoremstyle{definition}

\newtheorem{definition}[theorem]{Definition}
\newtheorem{remark}[theorem]{Remark}

\numberwithin{equation}{section}

\newcommand{\eps}{\varepsilon}

\newcommand{\F}{\mathcal F}
\newcommand{\Z}{\mathcal Z}
\newcommand{\N}{\mathbb N}
\newcommand{\potimes}{\,\widehat{\otimes}\,}
\newcommand{\Aff}{\mathrm{Aff}}
\newcommand{\U}{\mathcal U}
\newcommand{\V}{\mathcal V}
\newcommand{\BN}{\beta\mathbb N}
\newcommand{\Cu}{\mathrm{Cu}}
\newcommand{\Hs}{\mathcal H}
\newcommand{\BHs}{\mathcal B(\Hs)}
\title{$\mathcal{Z}$-stability and finite-dimensional tracial boundaries}

\author{Andrew Toms}
\address{\hskip-\parindent
Andrew Toms, Department of Mathematics, Purdue University, 150 North University Street, West Lafayette, IN 47907, USA.}
\email{atoms@pudue.edu}
\author{Stuart White}
\address{\hskip-\parindent
Stuart White, School of Mathematics and Statistics, University of Glasgow, 
University Gardens, Glasgow Q12 8QW, Scotland.}
\email{stuart.white@glasgow.ac.uk}
\author{Wilhelm Winter}
\address{\hskip-\parindent
Wilhelm Winter, Mathematisches Institut der WWU M\"unster, Einsteinstra\ss{}e 62, 48149, M\"unster, Germany.}
\email{wwinter@uni-muenster.de}

\thanks{Research partially supported by  EPSRC (grants No.\  EP/G014019/1 and No.\ EP/I019227/1), by the DFG (SFB 878) and by NSF (DMS-0969246). Andrew Toms is partially supported by the 2011 AMS Centennial Fellowship.}

\begin{document}
\begin{abstract}
We show that a  simple separable unital nuclear nonelementary $C^*$-algebra whose tracial state space has a compact extreme boundary with finite covering dimension admits uniformly tracially large order zero maps from matrix algebras into its central sequence algebra.  As a consequence, strict comparison implies $\Z$-stability for these algebras.  
\end{abstract}

\maketitle

\section{Introduction}

\noindent
The theme of tensorial absorption is prominent in the theory of operator algebras, particularly where the classification of these algebras is concerned.  An important step in Connes' proof that the hyperfinite $\mathrm{II}_1$ factor $\mathcal{R}$ is the unique injective $\mathrm{II}_1$ factor with separable predual exemplifies this theme:  such a factor, say $\mathcal{M}$, has the property that $\mathcal{M} \overline{\otimes} \mathcal{R} \cong \mathcal{M}$.  Another example arises in the theory of Kirchberg algebras, where the fact, due to Kirchberg, that such algebras absorb the Cuntz algebra $\mathcal{O}_\infty$ tensorially (see \cite{KP:Crelle}) plays a key role in their eventual classification via $\mathrm{K}$-theory.  For general nuclear separable $C^*$-algebras, the best tensorial absorption theorem that one can hope for is absorption of the Jiang-Su algebra $\mathcal{Z}$ ($\mathcal{Z}$-stability).  $\mathcal{Z}$-stability is a powerful tool for the classification of simple nuclear separable $C^*$-algebras, but is generally difficult to establish.  The property of strict comparison, on the other hand, is often  easier to verify.  Examples show that both properties, although by no means automatic, often occur at the same time, and are closely related to finite topological dimension. These and other considerations led the first and third named authors to conjecture the following connections between regularity properties for $C^*$-algebras:

\begin{conjecture}
Let $A$ be a simple nuclear separable unital infinite-dimensional $C^*$-algebra.  The following conditions are equivalent:
\begin{enumerate}
\item $A$ has finite nuclear dimension;
\item $A$ is $\mathcal{Z}$-stable;
\item $A$ has strict comparison.
\end{enumerate} 
\end{conjecture}

\noindent
The implications $(1) \Longrightarrow (2)$ and $(2) \Longrightarrow (3)$ have been established by the third named author and R\o rdam, respectively (see \cite{W:Invent2} and \cite{R:IJM}).  The reversal of either of these implications is at present only partial;  proving $(3) \Longrightarrow (2)$ becomes accessible if one additionally assumes certain local approximation and divisibility  properties \cite{W:Invent2} but at least the former assumption should ultimately be unnecessary.   In a recent breakthrough, Matui and Sato lifted the local approximation hypothesis, establishing $(3) \Longrightarrow (2)$ for algebras with finitely many extremal tracial states \cite{MS:Acta}.  Subsequently it has been an urgent task to remove this restriction on the tracial state space and in this article we extend their result to algebras whose extremal tracial boundary is compact and of finite covering dimension.  The problem has received substantial attention; this  result has also been  discovered by Kirchberg and R\o{}rdam \cite{KR:InPrep} and Sato \cite{S:Pre}.

A word on the idea of our proof is in order, as the details are necessarily technical.  A loose but in our case profitable way of thinking of a simple unital separable nuclear $C^*$-algebra $A$ with nonempty tracial state space is as a collection of everywhere nonzero ``sections'', where $A$ is viewed as a kind of noncommutative bundle over its space $T(A)$ of tracial states.  For a classical topological vector bundle $\xi$ over a space $X$ we have the property of local triviality:  restriction of the bundle to a sufficiently small neighborhood is a Cartesian product of the neighbourhood with a (complex) vector space having the same rank as $\xi$.  The complexity of $\xi$ is generated by the way in which these local trivialisations are patched together.  The analog of local trivialisation in our case comes from approximately central order zero maps with finite-dimensional domain which are large in a small open neighbourhood in $T(A)$.  Our objective is to construct an approximately central order zero map which is globally large over $T(A)$, and so we look to glue together the maps which work locally.  To do this we use nuclearity, and in particular the existence of approximate diagonals, to prove that there exist positive approximately central contractions in $A$ which, at the level of traces, represent indicator functions for open subsets of the extreme tracial boundary.  Indeed, any continuous strictly positive affine function on $T(A)$ is uniformly realised by positive elements in an approximately central manner (Lemma \ref{L:CT}).   Using suitable functions arising from open covers of $\partial_eT(A)$ we can patch together the local trivialisations to arrive at an order zero map which is uniformly bounded away from zero in trace. The bound we  obtain depends on the covering dimension of $\partial_eT(A)$ and it is here where finite dimensionality arises.  

Throughout the entire paper we work with central sequence algebras. The importance of these to tensorial absorption results dates back to McDuff's characterisation of separably acting II$_1$ factors which absorb the hyperfinite II$_1$ factor  \cite{McD:PLMS}.  In the $C^*$-setting central sequence algebras are intimately connected with with properties of stablity under tensoring by $\mathcal O_\infty$ and $\mathcal Z$, \cite{K:Able}.

The hypothesis of a compact extreme tracial boundary with finite covering dimension arose in \cite{DT:JFA}, where it was shown that for a simple algebra with strict comparison satisfying the said hypothesis, the Cuntz semigroup is almost divisible.  The main result in this paper can be viewed as a proof that this almost divisibility can occur in an almost central manner. 

Our paper is organised as follows. In Section \ref{Sect2} we introduce notation and review the relevant background from \cite{MS:Acta}. The main technical result (Lemma \ref{L:Key}) is established in Section \ref{Sect3}, and we show how the case of zero dimensional compact extreme tracial boundaries follows directly from this lemma. In the fourth and last section we extend to higher dimensional boundaries.  Our method for doing this differs from those in \cite{KR:InPrep,S:Pre} as we marry the ideas needed to extend \cite{MS:Acta} to the zero dimensional compact extremal case with the geometric sequence arguments developed by the third named author in \cite{W:Invent1,W:Invent2}.

\section{Uniformly tracially large order zero maps}\label{Sect2}

\noindent
Recall that a completely positive (cp) map $\phi:A\rightarrow B$ between $C^*$-algebras is said to be \emph{order zero} if it preserves orthogonality, i.e. if $e,f\in A_+$ have $ef=0$, then $\phi(e)\phi(f)=0$.  The structure theory of these maps is developed in \cite{WZ:MJM}, which in particular establishes a functional calculus. Given a completely positive and contractive (cpc) order zero map $\phi:A\rightarrow B$, and a positive contractive function $f\in C_0(0,1]$, there exists a cpc order zero map $f(\phi):A\rightarrow B$. For projections $p\in A$, this map satisfies 
\begin{equation}
f(\phi)(p)=f(\phi(p)).
\end{equation}
Secondly, given a cpc order zero map $A\rightarrow B$ and a tracial state $\tau:B\rightarrow\mathbb C$, the composition $\tau\circ\phi$ defines a positive tracial functional on $A$, \cite[Corollary 4.4]{WZ:MJM}. The other key fact we need is that order zero maps with finite dimensional domains are projective in the sense of Lemma \ref{L:Proj} below. This follows from the duality between order zero maps with domain $A$ and $^*$-homomorphisms from the cone $C(A)=C_0(0,1]\otimes A$ (see \cite{WZ:MJM}) and Loring's projectivity of cones over finite dimensional $C^*$-algebras \cite{L:Book}.  

\begin{lemma}\label{L:Proj}
Let $A,B,F$ be  $C^*$-algebras with $F$ finite dimensional and let $q:A\twoheadrightarrow B$ a surjective $^*$-homomorphism. Given a cpc order zero map $\phi:F\rightarrow B$, there exists a cpc order zero map $\tilde{\phi}:F\rightarrow A$ with $q\circ\tilde{\phi}=\phi$.
\end{lemma}

Given a $C^*$-algebra $A$, we denote the quotient $C^*$-algebra $\ell^\infty(A)/c_0(A)$ by $A_\infty$ and refer to this as the \emph{sequence algebra of $A$}. We have a natural $^*$-homomorphism from $A$ into $\ell^\infty(A)$ obtained by regarding each element of $A$ as a constant sequence in $\ell^\infty(A)$. Following this with the quotient map from $\ell^\infty(A)$ into $A_\infty$, we obtain an embedding $A\hookrightarrow A_\infty$, and we use this to regard $A$ as a $C^*$-subalgebra of $A_\infty$ henceforth.  In this way we can form the relative commutant $A_\infty\cap A'$, which is referred to as the \emph{central sequence algebra} of $A$. A bounded sequence $(x_n)_{n=1}^\infty$ in $A$ is said to be \emph{central} if its image $[(x_n)_{n=1}^\infty]$ in $A_\infty$ lies in $A_\infty\cap A'$.

We write $T(A)$ for the tracial state space of $A$.   Given a sequence $(\tau_n)_{n=1}^\infty$ in $T(A)$ and a free ultrafilter $\omega\in\BN\setminus\N$, the trace
$$
(x_n)_{n=1}^\infty\mapsto\lim_{n\rightarrow\omega}\tau_n(x_n)
$$
on $\ell^\infty(A)$ induces a trace on $A_\infty$. We write $T_\infty(A)$ for the collection of all traces on $A_\infty$ arising in this fashion.  When $\tau_n=\tau\in T(A)$ for all $n$, we write $\tau_\omega$ for the resulting trace in $T_\infty(A)$. We use the traces in $T_\infty(A)$ to define uniformly tracially large order zero maps into $A_\infty$.  

\begin{definition}\label{D:TLM}
Let $A$ be a separable unital $C^*$-algebra with $T(A)\neq\emptyset$.  A completely positive and contractive order zero map $\Phi:M_k\rightarrow A_\infty$ is \emph{uniformly tracially large} if $\tau(\Phi(1_k))=1$ for all $\tau\in T_\infty(A)$.
\end{definition}

By Lemma \ref{L:Proj}, every cpc order zero map $M_k\rightarrow A_\infty$ lifts to a cpc order zero map $M_k\rightarrow\ell^\infty(A)$ which consists of a sequence of cpc order zero maps $M_k\rightarrow A$.  We can rephrase the uniformly tracially large condition in terms of these liftings and traces on $A$. Indeed, the definition of $T_\infty(A)$ is designed to make it easy to manipulate conditions of the form (\ref{L:TI1}).

\begin{lemma}\label{L:TI}
Let $A$ be a separable unital $C^*$-algebra with $T(A)\neq\emptyset$.  Let   $\Phi:M_k\rightarrow A_\infty$ be a cpc order zero map. Then $\Phi$ is uniformly tracially large if and only if any lifting $(\phi_n):M_k\rightarrow\ell^\infty(A)$ of $\Phi$ to a sequence of cpc order zero maps satisfies 
\begin{equation}\label{L:TI1}
\lim_{n\rightarrow\infty}\min_{\tau\in T(A)}\tau(\phi_n(1_k))=1.
\end{equation}
\end{lemma}
\begin{proof}
That (\ref{L:TI1}) implies that $\Phi$ is uniformly tracially large is immediate.  For the converse, suppose $\Phi:M_k\rightarrow A_\infty$ is a uniformly tracially large cpc order zero map but (\ref{L:TI1}) fails for some  lifting $(\phi_n)$.  Then, there exists $\eps>0$, an increasing sequence $\{m_n\}_{n=1}^\infty$ in $\N$ and traces $\tau_n\in T(A)$ such that $\tau_n(\phi_{m_n}(1_k))\leq 1-\eps$ for all $n\in\N$.  Given a free ultrafilter $\omega$, the map $\rho:[(x_n)_{n=1}^\infty]\mapsto \lim_{n\rightarrow\omega}\tau_n(x_{m_n})$ defines a trace in $T_\infty(A)$, which has $\rho(\Phi(1_k))\leq 1-\eps$, contrary to hypothesis. 
\end{proof}

\begin{remark}\label{R:OT}
In a similar vein, for a fixed trace $\tau\in T(A)$ a cpc order zero map $\Phi:M_k\rightarrow A_\infty$ has $\tau_\omega(\Phi(1_k))=1$ for all $\omega\in\BN\setminus\N$ if and only if any lifting $(\phi_n)_n$ of $\Phi$ to a sequence of cpc order zero maps $M_k\rightarrow A$ has $\lim_{n\rightarrow\infty}\tau(\phi_n(1_k))=1$.
\end{remark}

Via functional calculus and a standard central sequence technique, uniformly tracially large cpc order zero maps $M_k\rightarrow A_\infty\cap A'$ give rise to the maps produced by \cite[Lemma 3.3]{MS:Acta}.

\begin{lemma}\label{L:MS33}
Let $A$ be a separable unital $C^*$-algebra with $T(A)\neq \emptyset$. Suppose that there exists a uniformly tracially large cpc order zero map $\Phi:M_k\rightarrow A_\infty\cap A'$. Then the conclusion of \cite[Lemma 3.3]{MS:Acta} holds for $A$. That is there exists a cpc order zero map $\Psi:M_k\rightarrow A_\infty\cap A'$ and a central sequence $(c_n)_{n=1}^\infty$ of positive contractions in $A$ such that
\begin{equation}\label{L:MS33:1}
\lim_{n\rightarrow\infty}\max_{\tau\in T(A)}|\tau(c_n^m)-1/k|=0
\end{equation}
for any $m\in\mathbb N$ and $\Psi(e)=[(c_n)_{n=1}^\infty]$ for some minimal projection $e\in M_k$.  
\end{lemma}
\begin{proof}
We need to produce a cpc order zero map $\Psi:M_k\rightarrow A_\infty\cap A'$ such that $\tau(\Psi^m(1_k))=1$ for each $m\in\N$ and $\tau\in T_\infty(A)$. Given such a map, fix a minimal projection $e\in M_k$ and take a lifting $(\psi_n)_{n=1}^\infty$ of $\Psi$ to a sequence of cpc order zero maps from $M_k$ to $A$. We can then set $c_n=\psi_n(e)$, so that $(c_n)_{n=1}^\infty$ is a central sequence. For each $m\in\N$ we have 
$$
\lim_{n\rightarrow\infty}\min_{\tau\in T(A)}\tau(\psi_n^m(1_k))=1
$$
by Lemma \ref{L:TI}. For each $m,n\in\N$ and $\tau\in T(A)$, the map $\tau(\psi_n^m(\cdot))$ is a trace on $M_k$ (\cite[Corollary 4.4]{WZ:MJM}), so $\tau(c_n^m)=\tau(\psi_n^m(1_k))/k$. Hence (\ref{L:MS33:1}) holds.

To construct $\Psi$, fix a uniformly tracially large cpc order zero map $\Phi:M_k\rightarrow A_\infty\cap A'$.  Then, for each $m\in\N$, the map $\Phi^{1/m}:M_k\rightarrow A_\infty\cap A'$ is a cpc order zero map. Lift each $\Phi^{1/m}$ to a sequence $(\phi^{(m)}_n)_{n=1}^\infty$ of cpc order zero maps $M_k\rightarrow A$. Fix a dense sequence $(x_r)_{r=1}^\infty$ in $A$ and for each $s\in\N$, we can find $r_s$ sufficiently large such that:
\begin{itemize}
\item $\|[\phi^{(s)}_{r_s}(y),x_i]\|\leq \frac{1}{s}\|y\|$, for all $y\in M_k$ and $i\in\{1,\dots,s\}$;
\item $\tau(\phi^{(s)}_{r_s}(1_k)^s)\geq 1-\frac{1}{s}$, for all $\tau\in T(A)$.
\end{itemize}
To obtain the second condition, note that $((\phi^{(s)}_n)^s)_{n=1}^\infty$ is a lifting of $\Phi$ and apply Lemma \ref{L:TI}.  The order zero map $\Psi:M_k\rightarrow A_\infty\cap A'$ induced by $(\phi^{(s)}_{r_s})_{s=1}^\infty$ has $\tau(\Psi^m(1_k))=1$ for all $m\in\N$ and $\tau\in T_\infty(A)$, as required. 
\end{proof}

The main result (Theorem 1.1) of \cite{MS:Acta} shows that for a simple separable unital nuclear nonelementary $C^*$-algebra $A$ with finitely many extremal traces and $T(A)\neq\emptyset$, the following properties are equivalent:
\begin{enumerate}[(i)]
\item $A$ is $\Z$-stable;
\item $A$ has strict comparison;
\item every completely positive map from $A$  to $A$ can be excised in small central sequences (see \cite[Definition 2.1]{MS:Acta} for the definition of this concept);
\item $A$ has  property (SI) as defined in \cite[Definition 3.3]{S:JFA} (see \cite[Definition 4.1]{MS:CMP} for the equivalent formulation used in \cite{MS:Acta}).
\end{enumerate}
The implication (i)$\implies$(ii) is due to R\o{}rdam \cite{R:IJM} and holds only assuming that $A$ is unital, separable, simple and exact. The implication (iii)$\implies$(iv) is immediate from the definitions. The proof of the remaining implications (ii)$\implies$(iii) and (iv)$\implies$(i) is valid for any unital simple separable nuclear $C^*$-algebra with $T(A)\neq\emptyset$ for which the conclusion of \cite[Lemma 3.3]{MS:Acta} holds as this is the only place where the extremal trace hypothesis enters play. For the implication (ii)$\implies$(iii) this is set out explicitly in the proof of \cite[Theorem 4.2]{MS:Acta}, and the proof of (iv)$\implies$(i) is readily seen to be a direct argument from property (SI) and the conclusion of \cite[Lemma 3.3]{MS:Acta}.  Using Lemma \ref{L:MS33}, we can formulate this result as follows.

\begin{theorem}[Matui-Sato]\label{MS}
Let $A$ be a simple separable unital nuclear  $C^*$-algebra with strict comparison. Suppose that for each $k\geq 2$, $A$ admits uniformly tracially large cpc order zero maps $M_k\rightarrow A_\infty\cap A'$.  Then $A$ is $\Z$-stable.
\end{theorem}

We now turn to amenability and Matui-Sato's construction of uniformly tracially large order zero maps for simple separable unital $C^*$-algebras with finitely many extremal traces. Recall that in \cite{H:Invent}, Haagerup showed that nuclear $C^*$-algebras are amenable in the sense of \cite{J:MAMS}. Moreover \cite[Theorem 3.1]{H:Invent} gives additional information on the location of a virtual diagonal witnessing amenability (\cite[Theorem 3.1]{H:Invent}). Combining this with Johnson's Hahn-Banach argument for extraction of an approximate diagonal from a virtual diagonal (\cite[Lemma 1.2]{J:AJM}) gives Lemma \ref{L.Amenable}, which is used in the proof of Sato's lemma below, as well as the construction of central sequences of positive elements with specified tracial behaviour in Section \ref{Sect3}.

\begin{lemma}[Haagerup]\label{L.Amenable}
Let $A$ be a unital nuclear $C^*$-algebra. Then for any finite subset $\F$ of $A$ and $\eta>0$, there exists $r\in\N$, contractions $a_1,\dots,a_r\in A$ and positive reals $\lambda_1,\dots,\lambda_r$ with $\sum_{i=1}^r\lambda_i=1$ such that:
\begin{enumerate}
\item\label{L.Amenable1} $\|\sum_{i=1}^r\lambda_i a_ia_i^*-1\|<\eta$;
\item\label{L.Amenable2} $\|\sum_{i=1}^r\lambda_i (xa_i\otimes a_i^*-a_i\otimes a_i^*x)\|_{A\potimes A}<\eta$ for all $x\in\F$,
\end{enumerate}
where $A\potimes A$ is the projective tensor product.
\end{lemma}

Let $N\subset\BHs$ be a von Neumann algebra acting on a separable Hilbert space $\Hs$.  We define $N^\infty$ to be the quotient $C^*$-algebra $\ell^\infty(N)/J$, where $J$ denotes the norm-closed two sided ideal of all strong$^*$-null sequences in $\ell^\infty(N)$.  Just as in the norm-closed setting, we can embed $N$ as a subalgebra of constant sequences in $N^\infty$ and so obtain the strong$^*$-central sequence algebra $N^\infty\cap N'$. With this notation we can now state the following result, which has been generalised to the nonnuclear setting in \cite{KR:InPrep}. 

\begin{lemma}[Sato, {\cite[Lemma 2.1]{S:arXiv}}]\label{L:Sato}Let $A$ be a  separable unital  nuclear  $C^*$-algebra. Suppose that $A\subset\BHs$ is a faithful unital representation of $A$ on a separable Hilbert space and write $N=A''$. Then the natural $^*$-homomorphism
$$
A_\infty\cap A'\rightarrow N^\infty\cap N'
$$
is surjective.  
\end{lemma}

Sato's lemma is the key ingredient in \cite[Lemma 3.3]{MS:Acta}, which obtains uniformly tracially large order zero maps when $A$ is separable, simple, unital and nuclear with finitely many extremal traces.  For use in Section \ref{Sect3}, we show how to deduce this from the previous lemma using projectivity of order zero maps in the context of a fixed extremal trace.  When $A$ has only finitely many extremal traces, a similar argument using the trace obtained from averaging the extremal traces can be used to prove \cite[Lemma 3.3]{MS:Acta}.  Recall that a II$_1$ factor $N$ is said to be \emph{McDuff} if it absorbs the hyperfinite II$_1$ factor $R$ tensorially, i.e. $N\cong N\,\overline{\otimes}\,R$. Every McDuff factor $N$ has an abundance of centralising sequences: for each $k\geq 2$, factorise $N\cong N\,\overline{\otimes}\,R$ and by regarding $R$ as the weak closure of the UHF-algebra $M_{k^\infty}$ we can consider the sequence of $n$-th tensor factor embeddings $M_k\rightarrow M_{k^\infty}$ to obtain a unital embedding $M_k\rightarrow N^\infty\cap N'$.

\begin{lemma}[cf. {\cite[Lemma 3.3]{MS:Acta}}]\label{L:OT}Let $A$ be a simple separable unital nuclear nonelementary $C^*$-algebra and let $\tau$ be an extremal tracial state on $A$.  For $k\geq 2$, there exists a cpc order zero map $\Phi:M_k\rightarrow A_\infty\cap A'$ with $\tau_\omega(\Phi(1_k))=1$ for all $\omega\in \BN\setminus\N$.
\end{lemma}
\begin{proof}
Fix $k\geq 2$. Let $\pi_\tau$ denote the GNS-representation associated to $\tau$.  As $\tau$ is an extremal trace, and $A$ is nuclear it follows that $\pi_\tau(A)''=N$ is an injective II$_1$ factor. By Connes' theorem \cite[Theorem 5.1]{C:Ann}, $N$ is McDuff.  As such, there is a unital embedding $\iota:M_k\hookrightarrow N^\infty\cap N'$.  By Lemma \ref{L:Sato} and the projectivity of order zero maps (Lemma \ref{L:Proj}), there exists an order zero map $\Phi:M_k\rightarrow A_\infty\cap A'$ lifting $\iota$.  For each $\omega\in\BN\setminus\N$, the trace $\tau_\omega$ given by $\tau_\omega((x_n)_{n=1}^\infty)=\lim_{n\rightarrow\omega}\tau(x_n)$ is well defined on $N^\infty$ and has $\tau_\omega(\iota(1_k))=1$.  Hence $\tau_\omega(\Phi(1_k))=\tau_\omega(\iota(1_k))=1$.
\end{proof}
\begin{remark}
Note that the map $\Phi$ of Lemma \ref{L:OT} already has $\tau_\omega(\Phi^m(1_k))=1$ for all $m\in\N$ and $\omega\in\BN\setminus\N$. We do not have to run the argument of Lemma \ref{L:MS33} to obtain this here.
\end{remark}

\section{Approximately central functions on the trace space}\label{Sect3}

\noindent
Recall that the tracial state space $T(A)$ of a separable unital $C^*$-algebra is a compact (in the weak$^*$-topology) convex subset of the state space of $A$, and so the Krein-Milman theorem shows that $T(A)$ is the closed convex hull of its extreme points $\partial_eT(A)$.  Further, $T(A)$ forms a metrisable Choquet simplex: every point of $T(A)$ is the barycentre of a unique measure supported on $\partial_eT(A)$, \cite{A:Book}. If additionally $\partial_eT(A)$ is compact, then $T(A)$ is known as a Bauer simplex. In this case we have a natural identification of $\Aff(T(A))=\{f:T(A)\rightarrow \mathbb R \mid f\text{ is continuous and affine}\}$ with $C_{\mathbb R}(\partial_eT(A))$ given by restriction (see \cite{G:Book}). Our objective in this section is Lemma \ref{L:Key} which enables us to produce a finite collection of cpc order zero maps with large sum when $T(A)$ has compact extreme boundary of finite covering dimension.  

The covering dimension of a compact Hausdorff space $X$ can be defined in a number of equivalent fashions (see \cite{P:Book}). We use the colouring formulation as follows. For $m\in\{0,1,\dots,\}$, say that $\dim X\leq m$ if and only if every finite open cover $\U$ of $X$ admits an $(m+1)$-colourable refinement $\V$: that is $\V$ is an open cover of $X$ with the property that every $V\in\V$ is contained in some element of $\U$ (i.e. $\V$ refines $\U$) and there exists a function $c:\V\rightarrow\{0,1,\ldots,m\}$ such that if $V,V'\in\V$ have $c(V)=c(V')$, then $V\cap V'=\emptyset$ (i.e. $\V$ can be $(m+1)$-coloured, in that each element of $\V$ can be assigned a colour such that two sets of the same colour are disjoint). We need a slight strengthening so that the sets in $\V$ form a closed cover of $X$. This is well known, but we include a proof for completeness. 

\begin{lemma}\label{L:CD}
Let $X$ be a compact Hausdorff topological space with $\dim(X)\leq m$.  Then for each finite open cover $\U$ of $X$, there exists a finite cover $\V$ consisting of closed sets refining $\U$ such that there is an $(m+1)$-colouring $c:\V\rightarrow\{0,1,\dots,m\}$ of $\V$ with the property that if $V,V'\in\V$ have $c(V)=c(V')$, then $V\cap V'=\emptyset$.
\end{lemma}
\begin{proof}
Given a finite open cover $\U$ of $X$, we can find an open cover $\tilde{\V}$ refining $\U$ which is $(m+1)$ colourable. Construct a partition of unity $(f_V)_{V\in\tilde{V}}$ subordinate to $\tilde{\V}$, i.e. $0\leq f_V\leq 1$ for all $V$, $\sum_{V\in \tilde{\V}}f_V(x)=1$ for all $x\in X$ and the support of each $f_V$ is contained in $V$. Let $\V$ be the collection of the supports of the $f_V$. This consists of closed sets, and refines $\tilde{\V}$ so is $(m+1)$-colourable. As every point $x\in X$ lies in the support of some $f_V$ it follows that $\V$ covers $X$.
\end{proof}

The following lemma of Lin (\cite{L:JFA}, based on  work of Cuntz and Pedersen \cite{CP:JFA}) enables us to realise strictly positive elements of $\Aff(T(A))$ via positive elements of $A$.

\begin{lemma}[Lin {\cite[Theorem 9.3]{L:JFA}}, following Cuntz, Pedersen {\cite{CP:JFA}}]\label{T.Lin}
Let $A$ be a simple  unital  $C^*$-algebra with non-empty tracial state space and let $f\in\Aff(T(A))$ be strictly positive. Then for any $\eps>0$, there exists $x\in A_+$ with $f(\tau)=\tau(x)$ for all $\tau\in T(A)$ and $\|x\|\leq\|f\|+\eps$.
\end{lemma}

Given any positive contraction $e$ in a nuclear $C^*$-algebra $A$ we can apply Haagerup's approximate diagonal to $e$ to produce a central sequence of positive contractions which has the same tracial behaviour as $e$. In particular, we can witness strictly positive elements of $\Aff(T(A))$ via central sequences of positive contractions.

\begin{lemma}\label{L:CT}
Let $A$ be a simple separable unital  nuclear $C^*$-algebra with a non-empty trace space and let $f$ be a positive affine continuous function on $T(A)$ with $\|f\|\leq 1$. Then there exists $(e_n)\in A_\infty\cap A'$  consisting of positive contractions in $A$ with
\begin{equation}\label{L:CT:E1}
\lim_{n\rightarrow\infty}\sup_{\tau\in T(A)}|\tau(e_n)-f(\tau)|=0.
\end{equation}
\end{lemma}
\begin{proof}
Define a sequence $(f_n)_{n=1}^\infty$ of continuous affine strictly positive functions on $T(A)$ by defining 
$$
f_n(\tau)=\frac{1}{3n}+\Big(1-\frac{2}{3n}\Big)f(\tau),
$$
for $\tau\in T(A)$. By construction each $f_n$ is strictly positive and has $\|f_n\|\leq 1-\frac{1}{3n}$ and $|f_n(\tau)-f(\tau)|\leq \frac{1}{n}$ for all $\tau\in T(A)$. For each $n\in\N$, take $\eps=\frac{1}{3n}$ in Lemma \ref{T.Lin} to obtain $x_n\in A_+$ with $\|x_n\|\leq 1$ such that $\tau(x_n)=f_n(\tau)$ for all $\tau\in T(A)$.  By Haagerup's theorem (Lemma \ref{L.Amenable}), we can find an approximate diagonal $(\sum_{i=1}^{l_n}\lambda^{(n)}_ia^{(n)}_i\otimes a^{(n)}_i{}^*)_{n=1}^\infty$ in $A\odot A$ such that each $\|a_i^{(n)}\|\leq 1$ and $\lambda^{(n)}_i$ are positive reals with $\sum_{i=1}^{l_n}\lambda^{(n)}_i=1$ for all $n$ and 
\begin{align}
&\Big\|\sum_{i=1}^{l_n}\lambda^{(n)}_ia_i^{(n)}a_i^{(n)}{}^*-1_A \Big\|\stackrel{n\rightarrow\infty}{\rightarrow}0;\label{L:CT1}\\
&\Big\|\sum_{i=1}^{l_n}\lambda^{(n)}_iba_i^{(n)}\otimes a_i^{(n)}{}^*-\sum_{i=1}^{l_n}\lambda^{(n)}_ia_i^{(n)}\otimes a_i^{(n)}{}^*b\Big\|_{A\potimes A}\stackrel{n\rightarrow\infty}{\rightarrow}0,\quad b\in A.\label{L:CT2}
\end{align}

Define $e_n=\sum_{i=1}^{l_n}\lambda^{(n)}_ia_i^{(n)}x_na^{(n)}_i{}^*$.  These are positive contractions in $A$.  For each $n\in\N$, the map $y\otimes z\mapsto yx_nz$ is contractive with respect to the projective tensor norm, so condition (\ref{L:CT2}) ensures that $(e_n)$ is a central sequence.  For $\tau\in T(A)$, we estimate
\begin{align*}
|\tau(x_n-e_n)|&=\Big|\tau\Big(x_n-\sum_{i=1}^{l_n}\lambda^{(n)}_ia_i^{(n)}x_na^{(n)}_i{}^*\Big)\Big|\\
&=\Big|\tau\Big(\Big(1_A-\sum_{i=1}^{l_n}\lambda^{(n)}_ia^{(n)}_i{}^*a_i^{(n)}\Big)x_n\Big)\Big|\\
&\leq\tau\Big(1_A-\sum_{i=1}^{l_n}\lambda^{(n)}_ia^{(n)}_i{}^*a_i^{(n)}\Big)\|x_n\|\\
&=\tau \Big(1_A-\sum_{i=1}^{l_n}\lambda^{(n)}_ia^{(n)}_ia_i^{(n)}{}^*\Big)\|x_n\|\\
&\leq \Big\|1_A-\sum_{i=1}^{l_n}\lambda^{(n)}_ia^{(n)}_ia_i^{(n)}{}^* \Big\|.
\end{align*}
Then (\ref{L:CT:E1}) follows from this estimate, (\ref{L:CT1}) and the fact that $|f(\tau)-\tau(x_n)|\leq \frac{1}{n}$ for all $\tau\in T(A)$.
\end{proof}

The next lemma enables us to convert central sequences which are tracially orthogonal to norm orthogonal sequences. The argument has its origins in Kishimoto's work \cite{K:JFA}, and our proof is based on \cite[Lemma 3.2]{MS:Acta}.  

\begin{lemma}\label{L:KT}
Let $A$ be a  separable unital $C^*$-algebra with non-empty trace space $T(A)$.  Let $T_0\subset T(A)$ be non-empty and suppose $(e^{(1)}_n)_{n=1}^\infty,\ldots,(e^{(L)}_n)_{n=1}^\infty$ are sequences of positive contractions in $A_+$ representing elements of $A_\infty\cap A'$ such that
\begin{equation}\label{KT:E1}
\lim_{n\rightarrow\infty}\sup_{\tau\in T_0}|\tau(e_n^{(l)}e_n^{(l')})|=0,\quad l\neq l'.
\end{equation}
Then there exist positive elements $\tilde{e}^{(l)}_n\leq e^{(l)}_n$ so that:
\begin{enumerate}[(i)]
\item $(\tilde{e}_n^{(l)})_n$ represents an element of $A_\infty\cap A'$;\label{KT:1}
\item $\lim_{n\rightarrow\infty}\sup_{\tau\in T_0}|\tau(\tilde{e}^{(l)}_n-e_n^{(l)})|=0$ for all $l$;\label{KT:2}
\item $(\tilde{e}_n^{(l)})_n\perp (\tilde{e}_n^{(l')})_n$ in $A_\infty\cap A'$ for $l\neq l'$.\label{KT:3}
\end{enumerate}
\end{lemma}
\begin{proof}
For each $l\in\{1,\dots,L\}$ and $n\in\N$ define
$$
g^{(l)}_n=(e_n^{(l)})^{1/2}\Big(\sum_{l'\neq l}e_n^{(l')}\Big)(e_n^{(l)})^{1/2},
$$
so $(g^{(l)}_n)_{n=1}^\infty$ is a central sequence for each $l$.  The hypothesis (\ref{KT:E1}) gives
$$
\sup_{\tau\in T_0}\tau(g_n^{(l)})\leq \sum_{l'\neq l}\sup_{\tau\in T(A)}\tau(e_n^{(l)}e_n^{(l')})\stackrel{n\rightarrow\infty}{\rightarrow}0.
$$

For $r\in\N$, define the continuous function $f_r:[0,\infty)\rightarrow [0,1]$ by $f_r(t)=\min(1,rt)$ and note that
$$
\inf_{t\geq 0}(1-f(t))t\leq 1/r.
$$
For $l\in\{1,\dots,L\}$ and $n,r\in\N$, define positive contractions by 
$$
x_{n,r}^{(l)}=(e_n^{(l)})^{1/2}\left(1-f_r(g_n^{(l)})\right)(e_n^{(l)})^{1/2}.
$$
These satisfy $x_{n,r}^{(l)}\leq e_n^{(l)}$ and for each $l$ and $r$, the sequence $(x_{n,r}^{(l)})_{n=1}^\infty$ represents an element of $A_\infty\cap A'$.  

For each $s\in\N$ and $l\in\{1,\dots,L\}$, we have
$$
\sup_{\tau\in T_0}\tau((g_n^{(l)})^s)\leq\|g_n^{(l)}\|^{s-1}\sup_{\tau\in T(A)}\tau(g_n^{(l)})\stackrel{n\rightarrow\infty}{\rightarrow}0.
$$
By choosing suitable polynomial approximations to $f_r(t)$ on $[0,L-1]$, it follows that
\begin{equation}\label{KT:E2}
\sup_{\tau\in T_0}\tau(e_n^{(l)}-x_{n,r}^{(l)})=\sup_{\tau\in T_0}\tau((e_n^{(l)})^{1/2}f_r(g_n^{(l)})(e_n^{(l)})^{1/2})\leq\|e^{(l)}_n\|\sup_{\tau\in T_0}\tau(f_r(g_n^{(l)}))\stackrel{n\rightarrow\infty}{\rightarrow}0,
\end{equation}
for each $l\in\{1,\dots,L\}$ and $r\in\N$. 

For each $l$, we compute exactly as in \cite[Lemma 3.2]{MS:Acta}, to obtain
\begin{align}
\left\|x_{n,r}^{(l')}x_{n,r}^{(l)}\right\|^2&=\left\|x_{n,r}^{(l)}(x_{n,r}^{(l')})^2x_{n,r}^{(l)}\right\|\nonumber\\
&\leq\Big\|x_{n,r}^{(l)}\Big(\sum_{j\neq l}x_{n,r}^{(j)}\Big)x_{n,r}^{(l)}\Big\|\nonumber\\
&=\Big\|(e_n^{(l)})^{1/2}\left(1-f_r(g_n^{(l)})\right)(e_n^{(l)})^{1/2}\Big(\sum_{j\neq l}x_{n,r}^{(j)}\Big)(e_n^{(l)})^{1/2}\left(1-f_r(g_n^{(l)})\right)(x_n^{(l)})^{1/2}\Big\|\nonumber\\
&=\big\|(e_n^{(l)})^{1/2}\left(1-f_r(g_n^{(l)})\right)g_n^{(l)}\left(1-f_r(g_n^{(l)})\right)(e_n^{(l)})^{1/2}\big\|\nonumber\\
&\leq\left\|\left(1-f_r(g_n^{(l)})\right)g_n^{(l)}\right\|\leq 1/r,\label{KT:E3}
\end{align}
for every $n,r\in\N$ and $l'\neq l$.  

Fix a countable dense sequence $(y_s)_{s=1}^\infty$ in $A$.  For each $r\in\N$, use (\ref{KT:E2}) and the fact that for each $l\in\{1,\dots,L\}$, $(x_{n,r}^{(l)})_{n=1}^\infty$ is a central sequence to obtain $N_r\in\N$ such that 
\begin{itemize}
\item $\|[x^{(l)}_{n,r},y_s]\|\leq 1/r$ for $s\in\{1,\dots,r\}$;
\item $\sup_{\tau\in T_0}\tau(e_{n}^{(l)}-x_{n,r}^{(l)})<1/r$
\end{itemize}
for $n\geq N_r$.  We may assume that $N_r<N_{r+1}$ for all $r$. Set $N_0=0$.  For $n\in\N$, let $r_n\in\N$ be such that $N_{r_n}<n\leq N_{r_{n+1}}$ so that $r_n\rightarrow\infty$ as $n\rightarrow\infty$ and define $\tilde{e}_n^{(l)}=x_{n,r_n}^{(l)}$.  The two conditions above give conditions (\ref{KT:1}) and (\ref{KT:2}), while (\ref{KT:3}) is a consequence of (\ref{KT:E3}).
\end{proof}

We are now in position to give the main technical lemma which is already enough to handle the $0$-dimensional compact extreme boundary case.

\begin{lemma}\label{L:Key}
Let $m\geq 0$, $k\geq 2$ and let $A$ be a simple separable unital  nuclear nonelementary $C^*$-algebra with $T(A)\neq\emptyset$ such that $\partial_eT(A)$ is compact with $\dim(\partial_eT(A))\leq m$. Then for each finite set $\F\subset A$ and $\eps>0$, there exist cpc order zero maps $\phi^{(0)},\dots,\phi^{(m)}:M_k\rightarrow A$ such that
\begin{equation}\label{L:Key:1}
\|[\phi^{(i)}(x),y]\|\leq \eps\|x\|,
\end{equation}
for all $i\in\{0,\dots,m\}$, $x\in M_k$, $y\in\F$ and such that for each $\tau\in\partial_eT(A)$, there exists $i(\tau)\in \{0,\dots,m\}$ such that $\tau(\phi^{(i(\tau))}(1_k))>1-\eps$.
\end{lemma}
\begin{proof}
Fix $k\geq 2$, a finite subset $\F\subset A$ and $\eps>0$.  For each $\tau\in \partial_eT(A)$, use Lemma \ref{L:OT} to provide a cpc order zero map $\Phi_\tau:M_k\rightarrow A_\infty\cap A'$ with $\tau_\omega(\Phi(1_k))=1$ for all $\omega\in\BN\setminus\N$. By Remark \ref{R:OT}, we can go sufficiently far down a sequence of cpc order zero maps from $M_k$ into $A$ which lift $\Phi_\tau$ to find a cpc order zero map
 $\phi_\tau:M_k\rightarrow A$ with $\tau(\phi_\tau(1_k))>1-\eps$ and
\begin{equation}\label{L:Key:6}
\|[\phi_\tau(x),y]\|<\eps\|x\|,
\end{equation}
for all $x\in M_k$, $y\in\F$, $\tau\in\partial_eT(A)$.   Define an open neighbourhood of $\tau$ in $\partial_eT(A)$ by $U_\tau=\{\rho\in\partial_eT(A):\rho(\phi_\tau(1_k))>1-\eps\}$.  Then $\partial_eT(A)=\bigcup_{\tau\in\partial_eT(A)}U_\tau$, so by compactness there exist $\tau_1,\dots,\tau_L\in\partial_eT(A)$ such that $\partial_eT(A)=\bigcup_{l=1}^LU_{\tau_l}$.  As $\dim(\partial_eT(A))\leq m$, Lemma \ref{L:CD} gives a finite cover $\V$ of $\partial_eT(A)$ consisting of closed sets such that $\V$ refines $\U=\{U_{\tau_1},\dots,U_{\tau_L}\}$ and a colouring $c:\V\rightarrow \{0,1,\dots,m\}$ such that if $c(V)=c(V')$, then $V$ and $V'$ are disjoint.  For each $i\in\{0,\dots,m\}$, write $\V^{(i)}=c^{-1}(\{i\})=\{V^{(i)}_1,\dots,V^{(i)}_{L_i}\}$.  

Fix $i\in\{0,\dots,m\}$.  For each $j=0,\dots,L_i$, choose a continuous function $f^{(i)}_j:\partial_eT(A)\rightarrow [0,1]$ on $\partial_eT(A)$ with $f^{(i)}_j=1$ on $V^{(i)}_j$ and $f^{(i)}_j=0$ on $\bigcup_{j'\neq j}V^{(i)}_{j'}$.  This is possible as elements of $\V^{(i)}$ are pairwise disjoint closed subsets of $\partial_eT(A)$. As $\partial_eT(A)$ is compact, we can extend $f^{(i)}_j$ to a continuous affine function on $T(A)$ also denoted $f^{(i)}_j$. Now apply Lemma \ref{L:CT} to obtain central sequences $(e_n^{(i,j)})_{n=1}^\infty$ of positive contractions such that
$$
\lim_{n\rightarrow\infty}\sup_{\tau\in T(A)}|\tau(e^{(i,j)}_n)-f^{(i)}_j(\tau)|=0,
$$
for each $j=1,\dots,L_i$.  Thus
\begin{equation}\label{L:Key:3}
\lim_{n\rightarrow\infty}\inf_{\tau\in \V^{(i)}_j}\tau(e^{(i,j)}_n)=1,\quad \lim_{n\rightarrow\infty}\sup_{\tau\in\bigcup_{j'\neq j}V^{(i)}_{j'}}\tau(e^{(i,j)}_n)=0.
\end{equation}
By construction
$$
\lim\sup_{\tau\in\bigcup_{s=1}^{L_i}V_s^{(i)}}\tau(e_n^{(i,j)}e_n^{(i,j')})=0
$$
for all $j\neq j'$ in $\{1,\dots,L_i\}$.  Therefore we can apply Lemma \ref{L:KT} with $T_0=\bigcup_{s=1}^{L_i}V^{(i)}_s$  to obtain central sequences $(\tilde{e}^{(i,j)}_n)_{n=1}^\infty$ of positive contractions with $\tilde{e}^{(i,j)}_n\leq e^{(i,j)}_n$,
\begin{equation}\label{L:Key:2}
\lim_{n\rightarrow\infty}\|\tilde{e}_n^{(i,j)}\tilde{e}_n^{(i,j')}\|=0
\end{equation}
for $j\neq j'$ and
\begin{equation}\label{L:Key:4}
\lim_{n\rightarrow\infty}\sup_{\tau\in\bigcup_{s=1}^{L_i}V_s^{(i)}}\tau(e^{(i,j)}_n-\tilde{e}^{(i,j)}_n)=0,
\end{equation}
for $j\in\{1,\dots,L_i\}$.  In this way (\ref{L:Key:3}) and (\ref{L:Key:4}) give
\begin{equation}\label{L:Key:5}
\lim_{n\rightarrow\infty}\inf_{\tau\in \V^{(i)}_j}\tau(\tilde{e}^{(i,j)}_n)=1.
\end{equation}

For $i\in\{0,\dots,m\}$ and $j\in\{1,\dots,L_i\}$, there exists $l(i,j)\in\{1,\dots,L\}$ such that $V^{(i)}_j\subset U_{\tau_{l(i,j)}}$. For $i\in\{0,\dots,m\}$ and $n\in\N$, define maps $\psi^{(i)}_n:M_k\rightarrow A$ by
\begin{equation}\label{L:Key:8}
\psi^{(i)}_n(x)=\sum_{j=1}^{L_i}\tilde{e}^{(i,j)}_n{}^{1/2}\phi_{\tau_{l(i,j)}}(x)\tilde{e}^{(i,j)}_n{}^{1/2},
\end{equation}
for $x\in M_k$. For each $i$, the sequences $(\psi^{(i)}_n)_{n=1}^\infty$ induce maps $\Psi^{(i)}:M_k\rightarrow A_\infty$.  Further, by (\ref{L:Key:2}), we have $(\tilde{e}^{(i,j)}_n)\perp (\tilde{e}^{(i,j')}_n)$ in $A_\infty\cap A'$, so that $\Psi^{(i)}$ is a sum of $L_i$ pairwise orthogonal cpc order zero maps and so is cpc and order zero. The condition $(\tilde{e}^{(i,j)}_n)\perp (\tilde{e}^{(i,j')}_n)$ in $A_\infty\cap A'$ also allows us to use (\ref{L:Key:6}) to obtain
\begin{equation}\label{L:Key:9}
\|[\Psi^{(i)}(x),y]\|\leq \max_{j\in\{1,\dots,L_i\}}\|[\phi_{\tau_{l(i,j)}}(x),y]\|<\eps\|x\|,
\end{equation}
for all $i\in\{0,\dots,m\}$, $x\in M_k$, $y\in\F$. For $\rho\in V^{(i)}_j$, we have $\rho(\phi_{\tau_{l(i,j)}}(1_k))>1-\eps$ as $V^{(i)}_j\subset U_{\tau_{l(i,j)}}$, giving
\begin{align}
\rho(\psi^{(i)}_n(1_k))&\geq \rho(\tilde{e}^{(i,j)}_n{}^{1/2}\phi_{\tau_{l(i,j)}}(1_k)\tilde{e}^{(i,j)}_n{}^{1/2})\nonumber\\
&=\rho(\tilde{e}^{(i,j)}_n\phi_{\tau_{l(i,j)}}(1_k))\nonumber\\
&=\rho(\phi_{\tau_{l(i,j)}}(1_k))-\rho((1_A-\tilde{e}^{(i,j)}_n)\phi_{\tau_{l(i,j)}}(1_k))\nonumber\\
&>(1-\eps)-\rho(1_A-\tilde{e}^{(i,j)}_n).\label{L:Key:7}
\end{align}
Combining (\ref{L:Key:7}) with (\ref{L:Key:5}) gives
\begin{equation}\label{L:Key:10}
\liminf_{n\rightarrow\infty}\inf_{\rho\in\bigcup_{j=1}^{L_i}V_j^{(i)}}\rho(\psi^{(i)}_n(1_k))>1-\eps.
\end{equation}

For each $i\in\{0,\dots,m\}$, take a lifting $(\phi^{(i)}_n)_{n=1}^\infty$ of $\Psi^{(i)}$ to a sequence of cpc order zero maps $M_k\rightarrow A$.  We claim that for $n$ sufficiently large, the maps $\phi^{(0)}_n,\dots,\phi^{(m)}_n$ satisfy the properties claimed in the statement of the lemma.  Indeed, since 
$$\sup_{\substack{x\in M_k\\\|x\|\leq 1}}\|\phi^{(i)}_n(x)-\psi^{(i)}_n(x)\|\rightarrow 0,$$ 
(\ref{L:Key:9}) gives
$$
\|[\phi^{(i)}_n(x),y]\|<\eps\|x\|
$$
for all $n$ sufficiently large and for all $i\in\{0,\dots,m\}$, $x\in M_k$, $y\in\F$.  By (\ref{L:Key:10}), we have
\begin{equation}\label{L:Key:11}
\liminf_{n\rightarrow\infty}\inf_{\rho\in\bigcup_{j=1}^{L_i}V_j^{(i)}}\rho(\phi^{(i)}_n(1_k))>1-\eps,
\end{equation}
and so for all $n$ sufficiently large we have
$$
\rho(\phi^{(i)}_n(1_k))>1-\eps,
$$
for all $i\in\{0,\dots,m\}$ and $\rho\in\bigcup_{j=1}^{L_i}V_j^{(i)}$ .   Since $\bigcup_{i=0}^m\bigcup_{j=1}^{L_i}V^{(i)}_j=\partial_eT(A)$, the result follows with $i(\rho)=\min\{i:\rho\in \bigcup_{j=1}^{L_i}V_j^{(i)}\}$.
\end{proof}

In the zero dimensional case, we immediately obtain uniformly tracially large order zero maps from the previous lemma.
\begin{theorem}\label{T:0D}
Let $A$ be a simple separable unital nuclear nonelementary $C^*$-algebra with $T(A)\neq\emptyset$ and $\partial_eT(A)$ compact and zero dimensional. Then for each $k\geq 2$, $A$ admits uniformly tracially large order zero maps $M_k\rightarrow A_\infty\cap A'$.
\end{theorem}
\begin{proof}
Fix $k\geq 2$. Take a nested sequence $(\F_n)_{n=1}^\infty$ of finite subsets of $A$ whose union is dense in $A$.  For each $n$, Lemma \ref{L:Key} gives a cpc order zero map $\phi_n:M_k\rightarrow A$ with
$$
\|[\phi_n(x),y]\|\leq \frac{1}{n}\|x\|,
$$
for all $y\in\F_n$ and $x\in M_k$, and
\begin{equation}\label{T:0D:1}
\tau(\phi_n(1_k))>1-\frac{1}{n}
\end{equation}
for all $\tau\in\partial_eT(A)$ and all $n\in\N$. By convexity, (\ref{T:0D:1}) holds for all $\tau\in T(A)$ and $n\in\N$. Thus the sequence $(\phi_n)_{n=1}^\infty$ induces a uniformly tracially large cpc order zero map $\Phi:M_k\rightarrow A_\infty\cap A'$ by Lemma \ref{L:TI}.
\end{proof}

\begin{corollary}
Let $A$ be a simple separable unital  nuclear nonelementary $C^*$-algebra with $T(A)\neq\emptyset$ and $\partial_eT(A)$ compact and zero dimensional.  Suppose $A$ has strict comparison, then $A$ is $\Z$-stable.
\end{corollary}
\begin{proof}
This follows immediately from Theorems \ref{MS} and \ref{T:0D}.
\end{proof}
\section{Higher dimensional compact extreme boundaries}\label{Sect4}

In this last section we extend the previous work to higher dimensional compact extreme boundaries.  The argument is based on the techniques developed in \cite{W:Invent1,W:Invent2}. The starting point is to use the finite dimensional compact extreme boundary to obtain a finite collection of cpc order zero maps with a large tracial sum in the sense of Lemma \ref{L:S1} below.  It is at this point in the argument where it is critical that we can obtain the family of order zero maps in Lemma \ref{L:Key} not just with large tracial sum but so that for each $\tau\in \partial_eT(A)$ one member of the family is large in $\tau$.

\begin{lemma}\label{L:S1}
Let $m\in\N$ and let $A$ be a simple  separable unital nuclear nonelementary $C^*$-algebra with $T(A)\neq\emptyset$ and $\partial_eT(A)$ is compact with $\dim(\partial_eT(A))\leq m$. Then, for $k\geq 2$ and any separable subspace $X\subset A_\infty$, there exist cpc order zero maps $\phi^{(0)},\dots,\phi^{(m)}:M_k\rightarrow A_\infty\cap A'\cap X'$ such that
\begin{equation}\label{L:S1:2}
\tau\Big(\sum_{i=0}^m\phi^{(i)}(1_k)b\Big)\geq\tau(b),
\end{equation}
for all $\tau\in T_\infty(A)$ and all $b\in (A_\infty)_+$.
\end{lemma}

\begin{proof}
Fix $k\geq 2$ and a separable subspace $X\subset A_\infty$.  Let $(y^{(i)})_{i=1}^\infty$ be a countable dense subset of $C^*(A,X)\subset A_\infty$ and lift each $y^{(i)}$ to a sequence $(y^{(i)}_n)_{n=1}^\infty$ in $\ell^\infty(A)$.  For $n\in\N$, define a finite subset of $A$ by $\F_n=\{y^{(i)}_m:1\leq i,m\leq n\}$.  For each $n$, use Lemma \ref{L:Key} to obtain cpc order zero maps $\phi^{(0)}_n,\dots,\phi^{(m)}_n:M_k\rightarrow A$ with
$$
\|[\phi^{(i)}_n(x),y]\|\leq \frac{1}{n}\|x\|,
$$
for $i\in\{0,\dots,m\}$, $x\in M_k$, $y\in\F_n$ and $n\in\N$ and such that for each $\rho\in \partial_eT(A)$ and $n\in\N$, there exists $i(\rho,n)$ with $\rho(\phi^{(i(\rho,n))}_n(1_k))>1-1/n$. These maps induce cpc order zero maps $\phi^{(0)},\dots,\phi^{(m)}:M_k\rightarrow A_\infty\cap A'\cap X'$.

Consider $\rho\in \partial_eT(A)$. For $n\in\N$, $a\in A_+$ with $\|a\|\leq 1$, we have
\begin{align}
\rho\Big(\sum_{i=0}^m\phi^{(i)}_n(1_k)a\Big)&\geq\rho(\phi_n^{(i(\rho,n))}(1_k)a)\nonumber\\
&=\rho(a)-\rho\Big(\big(1_A-\phi_n^{(i(\rho,n))}(1_k)\big)a\Big)\nonumber\\
&\geq\rho(a)-\rho(1_A-\phi_n^{(i(\rho,n))}(1_k))\nonumber\\
&\geq\rho(a)-\frac{1}{n}.\label{L:S1:1}
\end{align}
By convexity, the estimate (\ref{L:S1:1}) holds for all $\rho\in T(A)$.  Given a sequence $(\tau_n)_{n=1}^\infty$ in $T(A)$, a sequence $(b_n)_{n=1}^\infty$ of positive contractions in $A$ representing $b\in A_\infty$, and a free ultrafilter $\omega\in\BN\setminus\N$, taking limits in (\ref{L:S1:1}) gives
$$
\lim_{n\rightarrow\omega}\tau_n\Big(\sum_{i=0}^m\phi^{(i)}_n(1_k)b_n\Big)\geq\lim_{n\rightarrow\omega}\tau_n(b_n).
$$
That is
$$
\tau\Big(\sum_{i=0}^m\phi^{(i)}(1_k)b\Big)\geq\tau(b),
$$
for all $\tau\in T_\infty(A)$ and all $b\in (A_\infty)_+$, verifying (\ref{L:S1:2}). 
\end{proof}

Before proceeding, we extract a standard central sequence argument from \cite{W:Invent2}. Writing $\beta=\frac{1}{\overline{k}+1}$ and following the proof of \cite[Proposition 4.6]{W:Invent2} verbatim from equation (38) through to the 4th displayed equation on page 288 of \cite{W:Invent2}, one obtains the following lemma.
\begin{lemma}\label{W:Extract}
Let $A$ be a separable unital $C^*$-algebra, let $X\subset A_\infty$ be a separable subspace and let $0<\beta<1$. Suppose that for each $\eta>0$, there exist orthogonal positive contractions $d^{(0)}_\eta,d^{(1)}_\eta$ in $A_\infty\cap A'\cap X'$ with 
$$
\tau(d^{(i)}_\eta b)\geq \beta\tau(b)-\eta
$$
for $i\in\{0,1\}$, $\tau\in T_\infty(A)$ and contractions $b\in C^*(A,X)_+$.  Then there exist orthogonal positive contractions $d^{(0)},d^{(1)}\in A_\infty\cap A'\cap X'$ with
$$
\tau(d^{(i)}b)\geq\beta\tau(b),
$$
for $i=\{0,1\}$, $\tau\in T_\infty(A)$ and contractions $b\in C^*(A,X)_+$.
\end{lemma}

We can now use Lemma \ref{L:S1} to establish a version of \cite[Proposition 4.6]{W:Invent2}.  Essentially the argument is the same as the deduction of (36) and (37) from (33) in \cite{W:Invent2}, but since the maps arising in our proof have slightly different domains, we give the details for completeness.  For $0\leq\eta_1<\eta_2$, we denote by $g_{\eta_1,\eta_2}$ the continuous piecewise linear function on $\mathbb R$ given by
\begin{equation}\label{Defg}
g_{\eta_1,\eta_2}(t)=\begin{cases}1,&t\geq \eta_2;\\\frac{t-\eta_1}{\eta_2-\eta_1},&\eta_1<t<\eta_2;\\0,&t\leq\eta_1.\end{cases}
\end{equation}
\begin{lemma}\label{L:S2}
Given $m\geq 0$ and $k\geq 1$ there is $1\leq L_{m,k}\in\N$ such that, given a simple separable unital  nuclear nonelementary $C^*$-algebra $A$ with $T(A)\neq\emptyset$ such that $\partial_eT(A)$ is compact with $\dim(\partial_eT(A))\leq m$, and a separable subspace $X\subset A_\infty$, then there exist pairwise orthogonal contractions
$$
d^{(1)},\dots,d^{(k)}\in A_\infty\cap A'\cap X'
$$
such that
$$
\tau(d^{(i)}b)\geq \frac{1}{L_{m,k}}\tau(b)
$$
for all $i\in\{1,\dots,k\}$, $\tau\in T_\infty(A)$ and $b\in C^*(A,X)_+\subset A_\infty$.
\end{lemma}
\begin{proof}
When $k=1$, we can take $L_{m,k}=1$ and $d^{(1)}=1_{A_\infty}$.  We prove the statement when $k=2$.  Once the statement is established for $k=2$, the general case follows by induction using exactly the same argument as in the last two paragraphs of the proof of \cite[Proposition 4.6]{W:Invent2}.

Define $L_{m,2}=2(m+1)$ and fix a separable subspace $X\subset A_\infty$. By Lemma \ref{L:S1}, there exist cpc order zero maps $\phi^{(0)},\dots,\phi^{(m)}:M_{2(m+1)}\rightarrow A_\infty\cap A'\cap X'$ such that
\begin{equation}\label{L:S2:2}
\tau\Big(\sum_{i=0}^m\phi^{(i)}(1_{2(m+1)})b\Big)\geq\tau(b),
\end{equation}
for all $b\in (A_\infty)_+$ and $\tau\in T_\infty(A)$. For contractions $b\in C^*(A,X)_+$, the maps $\phi^{(i)}(\cdot)b$ are cpc and order zero, so $\tau(\phi^{(i)}(\cdot)b)$ is a trace on $M_{2(m+1)}$, \cite[Corollary 4.4]{WZ:MJM}. Thus
$$
\tau(\phi^{(i)}(e_{11})b)=\frac{1}{2(m+1)}\tau(\phi^{(i)}(1_{2m})b),
$$
for all $i\in\{0,\dots,m\}$, all contractions $b\in C^*(A,X)_+$ and all $\tau\in T_\infty(A)$. Summing over $i$, we have
\begin{equation}\label{L:S2:1}
\tau\Big(\sum_{i=0}^m\phi^{(i)}(e_{11})b\Big)=\frac{1}{2(m+1)}\tau\Big(\sum_{i=0}^m\phi^{(i)}(1_{2(m+1)})b\Big)\geq\frac{1}{2(m+1)}\tau(b),
\end{equation}
for all $b\in C^*(A,X)_+$ and all $\tau\in T_\infty(A)$.

For $\eta>0$, define 
$$
d_\eta^{(0)}=g_{\eta,2\eta}\Big(\sum_{i=0}^m\phi^{(i)}(e_{11})\Big),\quad d_\eta^{(1)}=1_{A_\infty}-g_{0,\eta}\Big(\sum_{i=0}^m\phi^{(i)}(e_{11})\Big)
$$
so that $d^{(0)}_\eta$ and $d^{(1)}_\eta$ are pairwise orthogonal positive contractions in $A_\infty\cap A'\cap X'$. We have $d^{(0)}_\eta+\eta1_A\geq \sum_{i=0}^m\phi^{(i)}(e_{11})$ so (\ref{L:S2:1}) gives
\begin{equation}\label{L:S2:3}
\tau(d^{(0)}_\eta b)\geq\frac{1}{2(m+1)}\tau(b)-\eta,
\end{equation}
for contractions $b\in C^*(A,X)_+$ and $\tau\in T_\infty(A)$. For a contraction $b\in C^*(A,X)_+$ and $\tau\in T_\infty(A)$ we have
\begin{align}
\tau((1_{A_\infty}-d^{(1)}_\eta)b)&=\tau\Big(g_{0,\eta}\Big(\sum_{i=0}^m\phi^{(i)}(e_{11})\Big)b\Big)\nonumber\\
&\leq\lim_{l\rightarrow\infty}\tau\Big(\Big(\sum_{i=0}^m\phi^{(i)}(e_{11})\Big)^{1/l} b\Big)\nonumber\\
&\leq\sum_{i=0}^m\lim_{l\rightarrow\infty}\tau\left((\phi^{(i)}(e_{11}))^{1/l}b\right)\label{L:S2:4}\\
&=\sum_{i=0}^m\lim_{l\rightarrow\infty}\tau\left((\phi^{(i)})^{1/l}(e_{11})b\right)\nonumber\\
&=\sum_{i=0}^m\lim_{l\rightarrow\infty}\frac{1}{2(m+1)}\tau\left((\phi^{(i)})^{1/l}(1_{2(m+1)})b\right)\label{L:S2:5}\\
&\leq\frac{m+1}{2(m+1)}\tau(b)=\frac{1}{2}\tau(b).\label{L:S2:6}
\end{align}
Here (\ref{L:S2:4}) uses the fact that $\langle \sum_{i=0}^m\phi^{(i)}(e_{11})\rangle\leq \sum_{i=0}^m\langle \phi^{(i)}(e_{11})\rangle$ in the Cuntz semigroup $\Cu(C^*(\phi^{(i)}(e_{11}):i=0,\dots,m))$ and $a\mapsto\lim_{l\rightarrow\infty}\tau(a^{1/l}b)$ is a dimension function on $C^*(\phi^{(i)}(e_{11}):i=0,\dots,m)\subset A_\infty\cap A'\cap X'$.  The equality (\ref{L:S2:5}) follows as for each $i$ and $l$, the map $\tau((\phi^{(i)})^{1/l}(\cdot)b)$ is a trace on $M_{2(m+1)}$, \cite[Corollary 4.4]{WZ:MJM}.  The estimate (\ref{L:S2:6}) gives
\begin{equation}\label{L:S2:7}
\tau(d^{(1)}_\eta b)\geq \Big(1-\frac{1}{2}\Big)\tau(b)=\frac{1}{2}\tau(b)
\end{equation}
for all $\tau\in T_\infty(A)$ and $b\in C^*(A,X)_+$.

Thus
$$
\tau(d^{(i)}_\eta b)\geq\frac{1}{2(m+1)}\tau(b)-\eta
$$
for $i=0,1$, $\tau\in T_\infty(A)$ and contractions $b\in C^*(A,X)_+$. As such the $k=2$ case of the lemma follows from Lemma \ref{W:Extract}.
\end{proof}

Combining the previous two lemmas we obtain order zero maps $M_k\rightarrow A_\infty\cap A'$ which are nowhere small in trace in the presence of compact finite dimensional extremal boundary.
\begin{proposition}\label{P:S3}
Given $m\geq 0$  there is $0<\alpha_{m}\leq 1$ such that, given $k\geq 2$, a simple separable unital  nuclear nonelementary $C^*$-algebra $A$ such that $T(A)\neq\emptyset$ and $\partial_eT(A)$ is compact with $\dim(\partial_eT(A))\leq m$, and a separable subspace $X\subset A_\infty$, there exists a cpc order zero map $\Phi:M_k\rightarrow A_\infty\cap A'\cap X'$ such that
$$
\tau(\Phi(1_k)b)\geq \alpha_{m}\tau(b),
$$
for all $\tau\in T_\infty(A)$ and $b\in C^*(A,X)_+\subset A_\infty$.
\end{proposition}
\begin{proof}
Fix $m\geq 0$, $k\geq 2$ and a separable subspace $X\subset A_\infty$. Suppose that $A$ is  simple separable unital  nuclear with $T(A)\neq\emptyset$ and $\partial_eT(A)$  compact and with $\dim(\partial_eT(A))\leq m$.  By Lemma \ref{L:S1}, there are cpc order zero maps $\phi^{(0)},\ldots,\phi^{(m)}:M_k\rightarrow A_\infty\cap A'\cap X'$ with
\begin{equation}\label{P:S3:1}
\tau\Big(\sum_{i=0}^m\phi^{(i)}(1_k)b\Big)\geq\tau(b),
\end{equation}
for all $\tau\in T_\infty(A)$ and all $b\in (A_\infty)_+$. By Lemma \ref{L:S2}, there exists $L_{m,m+1}\in\N$ and pairwise orthogonal contractions $d^{(0)},\dots,d^{(m)}$ in $A_\infty\cap A'\cap X'\cap (\mathrm{span}\{\phi^{(i)}(M_k):i=0,\dots,m\})'$ such that
\begin{equation}\label{P:S3:4}
\tau(d^{(i)}b)\geq \frac{1}{L_{m,m+1}}\tau(b),
\end{equation}
for all $i\in\{0,\dots,m\}$, $b\in C^*(A,X,\phi^{(j)}(M_k):j=0,\dots,m)_+$ and $\tau\in T_\infty(A)$.  

Define $\Phi:M_k\rightarrow A_\infty\cap A'\cap X'$ by
$$
\Phi(x)=\sum_{i=0}^m\phi^{(i)}(x)d^{(i)}.
$$
This is cpc and order zero as the $d^{(i)}$'s are pairwise orthogonal and commute with the image of the $\phi^{(j)}$'s. Given a contraction $b\in C^*(A,X)_+$ and $\tau\in T_\infty(A)$, we have
\begin{align}
\tau(\Phi(1_k)b)&=\tau\Big(\sum_{i=1}^m\phi^{(i)}(1_k)d^{(i)}b\Big)\nonumber\\
&\geq \frac{1}{L_{m,m+1}}\tau\Big(\sum_{i=1}^m\phi^{(i)}(1_k)b\Big)\label{P:S3:2}\\
&\geq\frac{1}{L_{m,m+1}}\tau(b)\label{P:S3:3},
\end{align}
where (\ref{P:S3:2}) follows from (\ref{P:S3:4}) and (\ref{P:S3:3}) from (\ref{P:S3:1}).  Thus we can take $\alpha_m=\frac{1}{L_{m,m+1}}$.
\end{proof}

A geometric sequence argument in the spirit of \cite{W:Invent1,W:Invent2}  can be used to obtain uniformly tracially large order zero maps from Proposition \ref{P:S3}.  The proof we give below uses the estimates from the geometric series argument in \cite[Lemma 5.11]{W:Invent2} but we simplify the calculations a little by taking a maximality approach.  We work abstractly from the conclusion of Proposition \ref{P:S3} and so begin by noting that this implies the conclusion of Lemma \ref{L:S2}. We continue to use the functions $g_{\eta_1,\eta_2}$ defined in (\ref{Defg}).

\begin{lemma}\label{L:GS}
Let $k\geq 2$ and suppose $A$ is a separable unital $C^*$-algebra with the property that there exists $\alpha>0$ with the property that for all separable subspaces $X\subset A_\infty$, there exists a cpc order zero map $\Phi:M_k\rightarrow A_\infty\cap A'\cap X'$ such that
\begin{equation}\label{L:GS:1}
\tau(\Phi(1_k)b)\geq \alpha\tau(b),
\end{equation}
for all $\tau\in T_\infty(A)$ and $b\in C^*(A,X)_+\subset A_\infty$. Then $A$ admits uniformly tracially large cpc order zero maps $M_k\rightarrow A_\infty\cap A'$.
\end{lemma}
\begin{proof}
Fix $k\geq 2$. First note that the hypothesis gives that there exists some $\gamma>0$ with the property that for any separable subspace $X\subset A_\infty$, there exist pairwise orthogonal positive contractions $d^{(0)},d^{(1)}\in A^\infty\cap A'\cap X'$ such that
\begin{equation}\label{L:GS:2}
\tau(d^{(i)}b)\geq \gamma\tau(b)
\end{equation}
for $i=\{0,1\}$, $\tau\in T_\infty(A)$ and $b\in C^*(A,X)_+\subset A_\infty$.  Indeed, one can take a cpc order zero map $\Phi:M_k\rightarrow A_\infty\cap A'\cap X'$ satisfying (\ref{L:GS:1}). For a contraction $b\in C^*(A,X)_+$, $\Phi(\cdot)b$ defines a cpc order zero map on $M_k$ so that $\tau(\Phi(\cdot)b)$ is a trace on $M_k$ for each $\tau\in T_\infty(A)$ (\cite[Corollary 4.4]{WZ:MJM}). As such
$$
\tau(\Phi(e)b)=\frac{\mathrm{Tr}_{M_k}(e)}{\mathrm{Tr}_{M_k}(1_k)}\tau(\Phi(1_k)b)\geq\frac{\mathrm{Tr}_{M_k}(e)}{\mathrm{Tr}_{M_k}(1_k)}\alpha\tau(b)
$$
for all $e\in (M_k)_+$, $\tau\in T_\infty(A)$ and $b\in C^*(A,X)_+$. Taking $d^{(i)}=\Phi(e^{(i)})$ for a pair $e^{(0)},e^{(1)}$ of orthogonal projections in $M_k$ of normalized trace at least $1/3$ we can take $\gamma=\alpha/3$.

Let $\alpha_0>0$ be the supremum of all $\alpha\geq 0$ with the property that for each separable subspace $X\subset A_\infty$, there exists a cpc order zero map $\Phi:M_k\rightarrow A_\infty\cap A'\cap X'$ satisfying (\ref{L:GS:1}).  We must prove that $\alpha_0=1$, so suppose to the contrary that $0<\alpha_0<1$. Fix a separable subspace $X\subset A_\infty$ and take $\eps>0$ such that
\begin{equation}\label{L:GS:6}
\alpha_1=\Big(\big(1-(\alpha_0-\eps)\gamma\big)(\alpha_0-3\eps)+(\alpha_0-\eps)\gamma\Big)>\alpha_0.
\end{equation}

Find a cpc order zero map $\Phi_0:M_k\rightarrow A_\infty\cap A'\cap X'$ such that
\begin{equation}\label{L:GS:4}
\tau(\Phi_0(1_k)b)\geq (\alpha_0-\eps)\tau(b)
\end{equation}
for all $\tau\in T_\infty(A)$ and $b\in C^*(A,X)_+\subset A_\infty$.  Find pairwise orthogonal contractions $d^{(0)},d^{(1)}\in A_\infty\cap A'\cap X'\cap \Phi_0(M_k)'$ such that 
\begin{equation}\label{L:GS:8}
\tau(d^{(i)}f(\Phi_0)(1_k)b)\geq \gamma\tau(f(\Phi_0)(1_k)b)
\end{equation}
for $i=\{0,1\}$, $f\in C_0(0,1]$, $\tau\in T_\infty(A)$ and $b\in C^*(A,X)_+$.

Define $$
\Phi_1(\cdot)=g_{2\eps,3\eps}(\Phi_0)(\cdot)+d^{(0)}\left(g_{\eps,2\eps}-g_{2\eps,3\eps}\right)(\Phi_0)(\cdot).
$$
This is certainly cpc. To see that $\Phi_1$ is order zero, note that for $x\in (M_k)_+$ we have
$$
\Phi_1(x)\leq g_{\eps,2\eps}(\Phi_0)(x)
$$
as $d^{(0)}$ commutes with $C^*(\Phi_0(M_k))$. Given $e,f\in (M_k)_+$ with $ef=0$, we have $$g_{\eps,2\eps}(\Phi_0)(e)^{1/2}g_{\eps,2\eps}(\Phi_0)(f)^{1/2}=0.$$ As $\Phi_1(e)^{1/2}\leq g_{\eps,2\eps}(\Phi_0)(e)^{1/2}$, there exists a sequence $(z_m)_{m=1}^\infty$ of contractions in $A$ with $z_mg_{\eps,2\eps}(\Phi_0)(e)^{1/2}\rightarrow \Phi_1(e)^{1/2}$ (see \cite[Lemma A-1]{H:MMJ}).  Similarly there exists a sequence of contractions $(w_m)_{m=1}^\infty$ in $A$ with $g_{\eps,2\eps}(\Phi_0)(f)^{1/2}w_m\rightarrow \Phi_1(f)^{1/2}$. As such
$$
\Phi_1(e)\Phi_1(f)=\lim_{m\rightarrow\infty}\Big(\Phi_1(e)^{1/2}z_mg_{\eps,2\eps}(\Phi_0)(e)^{1/2}g_{\eps,2\eps}(\Phi_0)(f)^{1/2}w_m\Phi_1(f)^{1/2}\Big)=0.
$$

Define
$$
h=d^{(1)}\left(g_{0,\eps}-g_{\eps,2\eps}\right)(\Phi_0)(1_k)+(1_{A_\infty}-g_{0,\eps}(\Phi_0)(1_k)).
$$
We have $d^{(1)}\perp d^{(0)}$ and $\left(g_{0,\eps}-g_{\eps,2\eps}\right)(\Phi_0)(1_k)\perp g_{2\eps,3\eps}(\Phi_0)(1_k)$ so that  $$d^{(1)}\left(g_{0,\eps}-g_{\eps,2\eps}\right)(\Phi_0)(1_k)\perp\Phi_1(1_k).$$ Also $(1_{A_\infty}-g_{0,\eps}(\Phi_0)(1_k))\perp \Phi_1(1_k)$ so that $h\perp\Phi_1(1_k)$.  

Now use the hypothesis again to find a cpc order zero map $\Phi_2:M_k\rightarrow A_\infty\cap A'\cap X'\cap\{h\}'$ such that
$$
\tau(\Phi_2(1_k)hb)\geq(\alpha_0-\eps)\tau(hb)
$$
for all $b\in C^*(A,X)_+$ and $\tau\in T_\infty(A)$.  The estimate (\ref{L:GS:8}) gives
\begin{eqnarray}
\tau(\Phi_2(1_k)hb)\geq (\alpha_0-\eps)\tau(hb)&\geq &(\alpha_0-\eps)\big(\gamma\tau\big((g_{0,\eps}-g_{\eps,2\eps})(\Phi_0)(1_k)b\big) \nonumber \\
&& +\tau\big((1_{A_\infty}-g_{0,\eps}(\Phi_0)(1_k))b\big)\big)\nonumber\\
&\geq &(\alpha_0-\eps)\gamma\tau\big((1_{A_\infty}-g_{\eps,2\eps}(\Phi_0)(1_k))b\big)\label{L:GS:3}
\end{eqnarray}
as $\gamma\leq 1$.

Define $\Psi:M_k\rightarrow A_\infty\cap A'\cap X'$ by $\Psi(x)=\Phi_1(x)+\Phi_2(x)h$. This is cpc and order zero as $\Phi_1(\cdot)$ and $\Phi_2(\cdot)h$ are cpc and order zero with orthogonal ranges. Then for $b\in C^*(A,X)_+$, we have
\begin{align}
\tau(\Psi(1_k)b)&=\tau(\Phi_1(1_k)b)+\tau(\Phi_2(1_k)hb)\nonumber\\
&\geq \tau\Big(\big(g_{2\eps,3\eps}(\Phi_0)(1_k)+d^{(0)}(g_{\eps,2\eps}-g_{2\eps,3\eps})(\Phi_0)(1_k)\big)b\Big)\nonumber\\
&\quad+(\alpha_0-\eps)\gamma\tau\big((1_{A_\infty}-g_{\eps,2\eps}(\Phi_0)(1_k))b\big)\nonumber\\
&\geq \tau\big(g_{2\eps,3\eps}(\Phi_0)(1_k)b\big)+(\alpha_0-\eps)\gamma\tau\big((1_{A_\infty}-g_{2\eps,3\eps}(\Phi_0)(1_k))b\big)\nonumber\\
&=\big(1-(\alpha_0-\eps)\gamma\big)\tau\big(g_{2\eps,3\eps}(\Phi_0)(1_k)b\big)+(\alpha_0-\eps)\gamma\tau(b)\label{L:GS:5}
\end{align}
using (\ref{L:GS:8}), (\ref{L:GS:3}) and the crude estimate $(\alpha_0-\eps)<1$.  As $g_{2\eps,3\eps}(\Phi_0)(1_k)\geq \Phi_0(1_k)-2\eps 1_A$, (\ref{L:GS:4}) gives
$$
\tau(g_{2\eps,3\eps}(\Phi_0)(1_kb))\geq (\alpha_0-3\eps)\tau(b)
$$
for all $b\in C^*(A,X)_+$.  Combining this with (\ref{L:GS:5}) gives
$$
\tau(\Psi(1_k)b)\geq \Big(\big(1-(\alpha_0-\eps)\gamma\big)(\alpha_0-3\eps)+(\alpha_0-\eps)\gamma\Big)\tau(b)=\alpha_1\tau(b)
$$
for all $b\in C^*(A,X)_+$.  The choice of $\eps$ in (\ref{L:GS:6}) ensures that $\alpha_1>\alpha_0$, giving the required contradiction.
\end{proof}

\begin{theorem}\label{T:UTL}
Let $A$ be a simple separable unital nuclear nonelementary  $C^*$-algebra with $T(A)\neq\emptyset$ and $\partial_eT(A)$ compact and $\dim(\partial_eT(A))<\infty$. Then for each $k\geq 2$, $A$ admits uniformly tracially large cpc order zero maps $M_k\rightarrow A_\infty\cap A'$.
\end{theorem}
\begin{proof}
This follows from Proposition \ref{P:S3} and Lemma \ref{L:GS}.
\end{proof}

\begin{corollary}\label{MainCor}
Let $A$ be a simple separable unital nuclear nonelementary $C^*$-algebra with $T(A)\neq\emptyset$ and $\partial_eT(A)$ compact and $\dim(\partial_eT(A))<\infty$. If $A$ has strict comparison, then $A$ is $\Z$-stable.
\end{corollary}
\begin{proof}
This follows from Theorems \ref{MS} and \ref{T:UTL}.
\end{proof}

\subsection*{Acknowledgements} Part of the research leading to Section \ref{Sect4} of this paper was carried out at the BIRS workshop `Descriptive set theory and functional analysis' in June 2011. We would like to record our gratitude to BIRS for the simulating research environment.  We would also like to thank Taylor Hines for his careful reading of an early version of this manuscript. 

\providecommand{\bysame}{\leavevmode\hbox to3em{\hrulefill}\thinspace}
\providecommand{\MR}{\relax\ifhmode\unskip\space\fi MR }
\providecommand{\MRhref}[2]{%
  \href{http://www.ams.org/mathscinet-getitem?mr=#1}{#2}
}
\providecommand{\href}[2]{#2}

\end{document}